\documentclass[a4paper,12pt]{amsart}%
\usepackage{amsfonts}
\usepackage{amsmath}
\usepackage{amssymb}
\usepackage{graphicx}%
\providecommand{\U}[1]{\protect\rule{.1in}{.1in}}
\newtheorem{theorem}{Theorem}
\theoremstyle{plain}

\newtheorem{definition}{Definition}

\numberwithin{equation}{section}
\begin{document}
\title[$G$-convergence for the nonlocal $p$-laplacian operator]{An elementary note about the $G$-convergence for a type of nonlocal
$p$-laplacian operator}
\author{Julio Mu\~{n}oz}
\maketitle

\begin{abstract}
Simple applications of a principle of minimum energy and the property of
monotonicity for the corresponding non-local $p$-laplacian operator, have
allowed a direct proof of the $G$-compactness in a weak sense, as well as in
the strong sense. The definitions of $H$-convergence and $G$-convergence, weak
or strong, have been proved to be equivalent.

\end{abstract}

\section{Introduction}

Let $\Omega$ be a bounded open and smooth domain in $\mathbb{R}^{N}$. We
consider $k,$ a radial and positive function in $L^{1}\left(  \mathbb{R}%
^{N}\right)  $ such that
\begin{equation}
\int_{\mathbb{R}^{N}\setminus B\left(  0,\epsilon\right)  }\frac{k\left(
\left\vert z\right\vert \right)  }{\left\vert z\right\vert ^{p}}dz<\infty,
\label{0}%
\end{equation}
where $1<p<+\infty,$ $B\left(  x,r\right)  $ denotes an open ball centered at
$x\in\mathbb{R}^{N}$ and radius $r>0,$ and $\epsilon$ is any positive number.
Additionally, $k$ is supposed to fulfill the estimation%
\begin{equation}
k\left(  \left\vert z\right\vert \right)  \geq\frac{c_{0}}{\left\vert
z\right\vert ^{N+\left(  s-1\right)  p}}\text{ for any }z\in\mathbb{R}%
^{N}\setminus\left\{  0\right\}  , \label{2}%
\end{equation}
where $c_{0}>0$ and $s\in\left(  0,1\right)  $ are given constants.

The natural frame in which we shall work is the nonlocal energy space
\begin{equation}
X=\left\{  u\in L^{p}\left(  \mathbb{R}^{N}\right)  :B\left(  u,u\right)
<\infty\right\}  \label{3}%
\end{equation}
where $B\left(  \cdot,\cdot\right)  $ is defined by means of the formula%
\begin{equation}
B\left(  u,v\right)  =\int_{\mathbb{R}^{N}}\int_{\mathbb{R}^{N}}k\left(
\left\vert x^{\prime}-x\right\vert \right)  \frac{\left\vert u\left(
x^{\prime}\right)  -u\left(  x\right)  \right\vert ^{p}}{\left\vert x^{\prime
}-x\right\vert ^{p}}dx^{\prime}dx. \label{4}%
\end{equation}
We define also the constrained energy space as%
\[
X_{0}=\left\{  u\in X:u=0\text{ a.e. in }\mathbb{R}^{N}\setminus
\Omega\right\}  .
\]
The space $X$ is a reflexive Banach space equipped with the norm
\[
\left\Vert u\right\Vert _{X}=\left\Vert u\right\Vert _{L^{p}\left(
\mathbb{R}^{N}\right)  }+\left(  B\left(  u,u\right)  \right)  ^{1/p}.
\]
The dual of $X$ will be denoted by $X^{\prime}\ $and can be endowed with the
norm
\[
\left\Vert g\right\Vert _{X^{\prime}}=\sup\left\{  \left\langle
g,w\right\rangle _{X^{\prime}\times X}:w\in X,\text{ }\left\Vert w\right\Vert
_{X}=1\right\}  .
\]
Analogous definitions applies for the space $X_{0}.$

There is another functional space that we will use in the formulation of our
problem, the one that is formed by the diffusion coefficients of our nonlocal
equations. That is
\[
\mathcal{H}\doteq\left\{  h\in L^{\infty}\left(  \mathbb{R}^{N}\times
\mathbb{R}^{N}\right)  \mid h\left(  x^{\prime},x\right)  =h\left(  x^{\prime
},x\right)  \in\lbrack h_{\min},h_{\max}]\text{ a.e. }\left(  x^{\prime
},x\right)  \in\mathbb{R}^{N}\times\mathbb{R}^{N}\right\}  ,
\]
where $h_{\min}\ $and $h_{\max}$ are positive given constants such that
$0<h_{\min}<h_{\max}.$

Now we are ready to set the nonlocal elliptic problem $\left(  P_{h}\right)
$: given $f\in X_{0}^{\prime}$ and $h\in\mathcal{H}$ we look for a function
$u\in X_{0}$ that solves the nonlocal boundary equation%
\begin{equation}
-2\left[  \text{\emph{p.v.}}%
{\displaystyle\int_{\mathbb{R}^{N}}}
h\left(  x^{\prime},x\right)  \frac{k\left(  \left\vert x^{\prime
}-x\right\vert \right)  \left\vert u\left(  x^{\prime}\right)  -u\left(
x\right)  \right\vert ^{p-2}\left(  u\left(  x^{\prime}\right)  -u\left(
x\right)  \right)  }{\left\vert x^{\prime}-x\right\vert ^{p}}dx^{\prime
}\right]  =f,\text{ }x\in\Omega. \label{P}%
\end{equation}
That means that, for any $v\in X_{0},$ $u\in X_{0}$ satisfies the variational
equality
\begin{equation}%
{\displaystyle\int_{\mathbb{R}^{N}}}
{\displaystyle\int_{\mathbb{R}^{N}}}
h\left(  x^{\prime},x\right)  \frac{k\left(  \left\vert x^{\prime
}-x\right\vert \right)  \left\vert u\left(  x^{\prime}\right)  -u\left(
x\right)  \right\vert ^{p-2}\left(  u\left(  x^{\prime}\right)  -u\left(
x\right)  \right)  \left(  v\left(  x^{\prime}\right)  -v\left(  x\right)
\right)  }{\left\vert x^{\prime}-x\right\vert ^{p}}dx^{\prime}dx=\left\langle
f,v\right\rangle _{X_{0}^{\prime}\times X_{0}} \label{problem}%
\end{equation}
The problem (\ref{P}) shall be written as%
\[
\left(  P_{h}\right)  \left\{
\begin{array}
[c]{c}%
\mathcal{L}_{h}u=f\text{ in }\Omega\smallskip\\
u=0\text{ in }\mathbb{R}^{N}\setminus\Omega.
\end{array}
\right.
\]
Throughout the manuscript, the action of $f$ upon a function $v\in X_{0}$ is
denoted by $\left\langle f,v\right\rangle ,$ and if $u$ is a solution of
$\left(  P_{h}\right)  ,$ then it can be expressed via the formula
\[
\left\langle \mathcal{L}_{h}u,v\right\rangle =B_{h}\left(  u,v\right)  ,
\]
where $B_{h}\left(  u,v\right)  $ is the left part of (\ref{problem}).

\begin{definition}
[$G$-convergence]Let $\left(  h_{j}\right)  _{j}$ be a sequence of
coefficients from $\mathcal{H}$ and let $\left(  u_{j}\right)  _{j}$ be the
sequence of the corresponding solution of the problems
\[
\left(  P_{j}\right)  \left\{
\begin{array}
[c]{c}%
\mathcal{L}_{h_{j}}u_{j}=f\text{ in }\Omega\smallskip\\
u_{j}=0\text{ in }\mathbb{R}^{N}\setminus\Omega.
\end{array}
\right.
\]
It is said that the sequence of operators $\left(  \mathcal{L}_{h_{j}}\right)
_{j}$ $G$-converges to the operator $\mathcal{L}_{h}$ when $j\rightarrow
\infty,$ if for any $f\in X_{0}^{\prime},$ $u_{j}\rightharpoonup u$ weakly in
$X_{0}$ if $j\rightarrow\infty,$ where $u$ is the solution to the problem
$\left(  P_{h}\right)  .$
\end{definition}

As we shall see, the problems $\left(  P_{h}\right)  $ or $\left(
P_{j}\right)  $ are well-posed (see Theorem \ref{Th00} below). This fact
allows us to rewrite the above definition as follows: for any $f\in
X_{0}^{\prime},$ $\mathcal{L}_{h_{j}}^{-1}f\rightharpoonup\mathcal{L}_{h}%
^{-1}f$ weakly in $X_{0}$ if $j\rightarrow\infty.$

We need to define the nonlocal flux associated to the operator $\mathcal{L}%
_{h}:$ it is
\[
\Psi_{h}\left(  x^{\prime},x\right)  =h\left(  x^{\prime},x\right)
k^{1/p^{\prime}}\left(  x^{\prime},x\right)  \frac{\left\vert u\left(
x^{\prime}\right)  -u\left(  x\right)  \right\vert ^{p-2}\left(  u\left(
x^{\prime}\right)  -u\left(  x\right)  \right)  }{\left\vert x^{\prime
}-x\right\vert ^{p-1}}\text{ where }u=\mathcal{L}_{h}^{-1}f
\]

\begin{definition}
[$H$-convergence]Let $\left(  h_{j}\right)  _{j}$ be a sequence of
coefficients from $\mathcal{H}$. It is said that the sequence of operators
$\left(  \mathcal{L}_{h_{j}}\right)  _{j}$ $H$-converges to the operator
$\mathcal{L}_{h}$ when $j\rightarrow\infty,$ if for any $f\in X_{0}^{\prime}$
the following convergences hold:

\begin{enumerate}
\item $u_{j}=\mathcal{L}_{h_{j}}^{-1}f\rightharpoonup u=\mathcal{L}_{h}^{-1}f
$ weakly in $X_{0}$ if $j\rightarrow\infty.$

\item $\Psi_{h_{j}}\rightharpoonup\Psi_{h}$ weakly in $L^{p^{\prime}}\left(
\mathbb{R}^{N}\times\mathbb{R}^{N}\right)  $ if $j\rightarrow\infty.$
\end{enumerate}
\end{definition}

The type of convergences we have just set up has been the aim of some recent
papers. In \cite{Bonder}, an $H$-convergence compactness result, via the
oscillating test function method of Tartar, is proved and, a characterization
result of $\Gamma$-convergence for the energy functional associated to the
problems $\left(  P_{j}\right)  ,$ is shown as well. In \cite{Bellido-Egrafov}%
, a characterization of the $H$-limit obtained in \cite{Bonder}, is given and,
the definitions of $G$-convergence and $H$-convergence have been proved to be
equivalent. \cite{Waurick} is a reference where these type convergences are
analyzed under a general abstract setting. In \cite{Du, Mengesha-Du},
interesting advances concerning with $\Gamma$-convergence are obtained, and in
\cite{Andres-Munoz}, a\ strong $G$-convergence result for the linear case
$p=2$ is obtained.

In the present manuscript, the set up of the problem is slightly different,
compared to some of the above references. It includes a kernel so that the
nonlocal problem is formulated in a space that may be smaller than
$W_{0}^{s,p}\left(  \Omega\right)  .$ As opposed to the detailed and nuanced
development elaborated in \cite{Bonder}, this work presents a procedure based
on much more elementary aspects. Here, the proof of the compactness in $L^{p}%
$, and even in stronger norms, have been performed in an simple way. The
departure point to develop the study of these convergences is an abstract
nonlocal energy criterion. This is a principle of minimum energy which,
combined with the monotony of the nonlocal operator, show in a transparent way
the steps we have to follow in order to give the proofs on compactness. The
exhibited procedure has allowed understanding the reason why, in this case,
$G$-convergence, in the weak sense, is equivalent to $G$-convergence in the
strong sense or $H$-convergence (see \cite{Jikov, Allaire} and \cite{Pankov}
for a general local setting).\smallskip

We organize the paper as follows: Section \ref{S2} contains some essential
notes related to the continuity and some fundamental compactness embeddings.
It also includes an essential result ensuring the well-posedness of the
nonlocal variational problems $\left(  P_{j}\right)  $ (Theorem \ref{Th00}).
In Section \ref{S3} Theorem \ref{Th0} establishes the relative convergence of
$\left(  \mathcal{L}_{h_{j}}^{-1}f\right)  \ $towards $\mathcal{L}_{h}^{-1}f $
strongly in the $L^{p}$-norm. This latter result is improved in Section
\ref{S4} by an strong $X_{0}$-compactness result (Theorem \ref{Th1}). After,
the equivalences among the definitions of convergence in $X_{0},$
$G$-convergence and $H$-convergence, of $\left(  \mathcal{L}_{h_{j}}\right)
_{j}$ towards $\mathcal{L}_{h},$ have been proved (Theorems \ref{Th2} and
\ref{Th3}). Section \ref{S5} is devoted to a brief summary of the procedures
and results we have obtained throughout the manuscript.

\section{Some preliminaries\label{S2}}

Some previous remarks concerning the ambient space on which we shall work,
have to be commented:

\begin{enumerate}
\item We firstly notice $X_{0}\subset W_{0}^{s,p}\left(  \Omega\right)  ,$
where $W_{0}^{s,p}\left(  \Omega\right)  =\overline{C_{c}^{\infty}\left(
\Omega\right)  }\ $and the closure is considered with respect to the classical
norm of $W^{s,p}\left(  \mathbb{R}^{N}\right)  .$ The above embedding is true
because in the case of $\Omega$ has a Lipschitz boundary, $W_{0}^{s,p}\left(
\Omega\right)  $ can be identified as the set functions form $W^{s,p}\left(
\mathbb{R}^{N}\right)  $ vanishing outside of $\Omega$ (see
\cite{DiNezza-Palatucci-Valdinoci}), that is
\[
W_{0}^{s,p}\left(  \Omega\right)  =\left\{  g\in W^{s,p}\left(  \mathbb{R}%
^{N}\right)  :g=0\text{ a.e. in }\mathbb{R}^{N}\setminus\Omega\right\}  .
\]
An important fact is the continuous embedding of $X_{0}$ into $L^{p}\left(
\Omega\right)  .$ Indeed, it is known that there exists a constant $c=c\left(
N,s\right)  >0$ such that for any $w\in W_{0}^{s,p}\left(  \Omega\right)  $%
\begin{equation}
c\left\Vert w\right\Vert _{L^{p}\left(  \Omega\right)  }^{p}\leq\int_{\Omega
}\int_{\Omega}\frac{\left\vert w\left(  x^{\prime}\right)  -w\left(  x\right)
\right\vert ^{p}}{\left\vert x^{\prime}-x\right\vert ^{N+sp}}dx^{\prime}dx
\label{Pre0}%
\end{equation}
(see \cite[Th. 6.7]{DiNezza-Palatucci-Valdinoci}). By paying attention to the
hypotheses on the kernel (\ref{2}), and using (\ref{Pre0}), we ensure there is
a positive constant $C$ such that the nonlocal Poincar\'{e} inequality
\begin{equation}
C\left\Vert w\right\Vert _{L^{p}\left(  \Omega\right)  }^{p}\leq B\left(
w,w\right)  \label{Prel3}%
\end{equation}
holds for any $w\in X_{0}.$ In particular, the above inequality implies that
$X_{0}$ is a Banach space when it is equipped with the (equivalent) norm
defined as $\left\Vert w\right\Vert _{X_{0}}=\left(  B\left(  w,w\right)
\right)  ^{\frac{1}{p}}.$

\item \label{point2}The embedding $X_{0}\subset L^{p}\left(  \Omega\right)  $
is compact. To make sure of that we simple take into account (\ref{Prel3}),
the compact embedding $W_{0}^{s,p}\left(  \Omega\right)  \subset L^{p}\left(
\Omega\right)  $ (see \cite[Th. 7.1]{DiNezza-Palatucci-Valdinoci})$\ $and the
fact that $X_{0}$ is closed with respect to strong convergence in $L^{p}$. In
particular, if $\left(  w_{j}\right)  _{j}\ $is a sequence uniformly bounded
in $X_{0},$ that is, if there is a constant $C$ such that $B_{h}\left(
w_{j},w_{j}\right)  \leq C$ for every $j,$ then the positiveness of the
coefficients and (\ref{Prel3}) guarantee $\left(  w_{j}\right)  _{j}$ is
uniformly bounded $L^{p}\left(  \Omega\right)  .$ Moreover, the compact
embedding $X_{0}\subset L^{p}\left(  \Omega\right)  $ ensures the existence of
a subsequence of $\left(  w_{j}\right)  _{j},$ still denoted by $\left(
w_{j}\right)  _{j},$ such that $w_{j}\rightarrow w\in X_{0}$ strongly in
$L^{p}\left(  \Omega\right)  .$

\item Another essential issue of our research is the identification of the
above nonlocal boundary problem with a Dirichlet Principle.

\begin{theorem}
\label{Th00}Let $h$ be any given function form $\mathcal{H}.$ There exists a
unique function $u\in X_{0}$ solution to the problem (\ref{problem}). This
function $u$ is the only solution to the minimization problem%
\begin{equation}
\min_{w\in X_{0}}J_{h}\left(  w\right)  \label{Dirichlet}%
\end{equation}
where%
\[
J_{h}\left(  w\right)  \doteq\frac{1}{p}\int_{\mathbb{R}^{N}}\int
_{\mathbb{R}^{N}}h(x^{\prime},x)k\left(  \left\vert x^{\prime}-x\right\vert
\right)  \frac{\left\vert w\left(  x^{\prime}\right)  -w\left(  x\right)
\right\vert ^{p}}{\left\vert x^{\prime}-x\right\vert ^{p}}dx^{\prime
}dx-\left\langle f\left(  x\right)  ,w\left(  x\right)  \right\rangle .
\]

\end{theorem}

See \cite{Bonder} for the proof.
\end{enumerate}

\section{Nonlocal energy criterion\label{S3}}

Let us examine how is the convergence of the sequence $\left(  u_{j}\right)
_{j}$ of solutions of the problem $\left(  P_{j}\right)  .$ We shall pay
attention to the fact that $u_{j}$ is the minimizer of $J_{h_{j}}\left(
w\right)  =\frac{1}{p}B_{h_{j}}\left(  w,w\right)  -\left\langle
f,w\right\rangle ,$ $w\in X_{0}.$ Now, going into the details, we firstly note
that we can extract a subsequence form $\left(  h_{j}\right)  _{j},$ still
denoted by $\left(  h_{j}\right)  _{j},$ such that $h_{j}\rightharpoonup h$
weakly-$\ast$ in $L^{\infty}\left(  \mathbb{R}^{N}\times\mathbb{R}^{N}\right)
.$ In addition, it is a well-known result that $h\in\mathcal{H}$. Let $u$ be
solution of $\left(  P_{h}\right)  ,$ the one that corresponds to the
coefficient $h$. Necessarily $u$ must be the minimizer of $J_{h}\left(
w\right)  =\frac{1}{p}B_{h}\left(  w,w\right)  -\left\langle f,w\right\rangle
,$ $w\in X_{0}.$

\begin{theorem}
\label{Th0}Under the above circumstances, there exists a subsequence of
$\left(  u_{j}\right)  _{j}$, still denoted by $\left(  u_{j}\right)  _{j},$
for which
\[
u_{j}\rightarrow u\text{ strongly in }L^{p}\left(  \mathbb{R}^{N}\right)  ,
\]
and%
\begin{equation}
\displaystyle\lim_{j}\min_{w\in X_{0}}\left\{  \frac{1}{p}B_{h_{j}}\left(
w,w\right)  -\left\langle f,w\right\rangle \right\}  =\min_{w\in X_{0}%
}\left\{  \frac{1}{p}B_{h}\left(  w,w\right)  -\left\langle f,w\right\rangle
\right\}  \label{G-con-energies}%
\end{equation}
In particular%
\begin{equation}
\lim_{j}B_{h_{j}}\left(  u_{j},u_{j}\right)  =B_{h}\left(  u,u\right)  .
\label{G-con-energies2}%
\end{equation}

\end{theorem}

\begin{proof}
The first consequence derived from the fact $u_{j}$ is a solution to $\left(
P_{j}\right)  $ is%
\[
C\left\Vert u_{j}\right\Vert _{L^{p}\left(  \Omega\right)  }^{p}\leq h_{\max}%
{\displaystyle\int_{\mathbb{R}^{N}}}
{\displaystyle\int_{\mathbb{R}^{N}}}
\frac{k\left(  \left\vert x^{\prime}-x\right\vert \right)  \left\vert
u_{j}\left(  x^{\prime}\right)  -u_{j}\left(  x\right)  \right\vert ^{p}%
}{\left\vert x^{\prime}-x\right\vert ^{p}}dx^{\prime}dx\leq\left\Vert
f\right\Vert _{X_{0}^{\prime}}\left\Vert u_{j}\right\Vert _{X_{0}}%
\]
for any $j.$ This implies the sequences of norms $\left\Vert u_{j}\right\Vert
_{X_{0}}$ and $\left\Vert u_{j}\right\Vert _{L^{p}\left(  \mathbb{R}%
^{N}\right)  }^{p}$ are uniformly bounded. This preliminary serves to ensure
that $\left(  u_{j}\right)  _{j},$ or at least for a subsequence of it, is
weakly convergent in $L^{p}\left(  \mathbb{R}^{N}\right)  $ to a function
$u^{\ast}$ $\in L^{p}\left(  \mathbb{R}^{N}\right)  .$ But, by exploiting the
fact that $u_{j}$ is also a sequence uniformly bounded in $X_{0},$ we improve
this convergence (see comment \ref{point2} in Section \ref{S2}):
$u_{j}\rightarrow u^{\ast}\in X_{0}$ strongly in $L^{p}\left(  \Omega\right)
.$ Let $m_{j}$ be the minimum value of functional $J_{h_{j}},$ which is
attained by $u_{j},\ $and let $m$ be the minimum value of $J_{h},$ which is
attained by $u.$ Since $h_{j}\rightharpoonup h$ weakly-$\ast$ in $L^{\infty
}\left(  \mathbb{R}^{N}\times\mathbb{R}^{N}\right)  $ and $\frac{k\left(
\left\vert x^{\prime}-x\right\vert \right)  \left\vert u\left(  x^{\prime
}\right)  -u\left(  x\right)  \right\vert ^{p}}{\left\vert x^{\prime
}-x\right\vert ^{p}}\in L^{1}\left(  \mathbb{R}^{N}\times\mathbb{R}%
^{N}\right)  $ then
\[
\lim_{j}m_{j}\leq\lim_{j}\frac{1}{p}B_{h_{j}}\left(  u,u\right)  -\left\langle
f,u\right\rangle =\frac{1}{p}B_{h}\left(  u,u\right)  -\left\langle
f,u\right\rangle =m.
\]
And the reverse is also fulfilled:%
\[
m\leq\frac{1}{p}B_{h}\left(  u^{\ast},u^{\ast}\right)  -\left\langle
f,u^{\ast}\right\rangle \leq\lim_{j}\frac{1}{p}B_{h_{j}}\left(  u_{j}%
,u_{j}\right)  -\left\langle f,u_{j}\right\rangle =\lim_{j}m_{j}%
\]
Note that to justify the second inequality we have used the following
argument: since the sequence of functions $\frac{k\left(  \left\vert
x^{\prime}-x\right\vert \right)  \left\vert u_{j}\left(  x^{\prime}\right)
-u_{j}\left(  x\right)  \right\vert ^{p}}{\left\vert x^{\prime}-x\right\vert
^{p}}$ converges a.e. in $\mathbb{R}^{N}\times\mathbb{R}^{N}$ to the function
$\frac{k\left(  \left\vert x^{\prime}-x\right\vert \right)  \left\vert
u\left(  x^{\prime}\right)  -u\left(  x\right)  \right\vert ^{p}}{\left\vert
x^{\prime}-x\right\vert ^{p}},$ and $h_{j}\rightharpoonup h$ weakly-$\ast$ in
$L^{\infty}\left(  \mathbb{R}^{N}\times\mathbb{R}^{N}\right)  ,$ then thanks
to a generalized version of Fatou's Lemma (see \cite{Royden}) we derive
$\lim_{j}B_{h_{j}}\left(  u_{j},u_{j}\right)  \geq B_{h}\left(  u,u\right)  $
(see \cite[Proposition 1, p.309]{Andres-Munoz}).\newline By linking the above
inequalities, we prove $\lim_{j}m_{j}=m$, which is factually
(\ref{G-con-energies}). This convergence implies that both $u$ and $u^{\ast}$
are solutions to $\left(  P_{h}\right)  ,$ so that, thanks to the uniqueness
of solution for $\left(  P_{h}\right)  $ (Theorem \ref{Th00})$,$ $u=u^{\ast}.$
In addition the limit (\ref{G-con-energies2}) holds.
\end{proof}

\section{Strong convergence in $X_{0}$\label{S4}}

There are two ingredients for improving the strong convergence in $L^{p}$. One
of them is the monotonicity of the operator $B\left(  \cdot,\cdot\right)  ,$
which is a consequence of the elementary inequality (\ref{ele_ine}) (see
below). More specifically, the fact that (\ref{ele_ine}) makes the sequence of
operators $\left(  \mathcal{L}_{h_{j}}^{-1}\right)  _{j}$ to be uniformly
strictly monotone. The second aspect is about the hypothesis concerning the
definition of the operator $B$: any of the singular integrals involved in the
above definitions, have been understood as the principal value. Compare to the
situation of pure convergent integral-Lebesgue, this assumption entails a
different treatment of the (\textit{a priori} nonpositive) singular integrals
involved in our analysis. This formulation of the problem allows us to prove a
certain formula of integration by parts that eventually, together the strong
convergence obtained in Theorem \ref{Th0}, show us, with clarity, the way to
achieve strong convergence in $X_{0}.$ However, this assumption concerning the
way we have to interpret the singular integrals is not necessary when we deal
with the case $p=2.$ In such a case the bilinearity of $B$ and the convergence
of energies obtained in (\ref{G-con-energies2}) are enough to prove the
aforementioned strong convergence in $X_{0}$ (see \cite{Andres-Munoz}).

\subsection{Strong convergence in $X_{0}$}

We are going to prove a $X_{0}$ strong compactness result. That simply means
that from the sequence $\left(  \mathcal{L}_{h_{j}}\right)  _{j}$ we can
extract a subsequence, $\left(  \mathcal{L}_{h_{j_{k}}}\right)  _{k},$ such
that $\mathcal{L}_{h_{j_{k}}}^{-1}f\rightarrow\mathcal{L}_{h}^{-1}f$ strongly
in the norm of $X_{0}$ if $k\rightarrow\infty.$

For convenience, in the remainder of the manuscript, we will adopt the
notation $v^{\prime}=v\left(  x^{\prime}\right)  $ for any function
$v=v\left(  x\right)  .$

\begin{theorem}
\label{Th1}From the given sequence of coefficients $\left(  h_{j}\right)
_{j}\subset\mathcal{H},$ we can extract a subsequence, which will not be
relabelled, such that $h_{j}\rightharpoonup h$ weak-$\ast$ in $L^{\infty
}\left(  \mathbb{R}^{N}\times\mathbb{R}^{N}\right)  $ and $\left(
u_{j}\right)  _{j}$ strongly converges to $u$ strongly in $X_{0}$ if
$j\rightarrow\infty.$
\end{theorem}

\begin{proof}
In a first stage we assume $p\geq2$. We define the sequences $a_{j}%
=u_{j}^{\prime}-u_{j}$ and $b=u^{\prime}-u$, and we take into account the next
elementary inequality: if $1<p<\infty,$ then there exist two positive
constants $C=C\left(  p\right)  $ and $c=c\left(  p\right)  ,$ such that for
every $a,$ $b\in\mathbb{R}^{N}$%
\begin{equation}
c\left\{  \left\vert a\right\vert +\left\vert b\right\vert \right\}
^{p-2}\left\vert a-b\right\vert ^{2}\leq\left(  \left\vert a\right\vert
^{p-2}a-\left\vert b\right\vert ^{p-2}b\right)  \cdot\left(  a-b\right)  \leq
C\left\{  \left\vert a\right\vert +\left\vert b\right\vert \right\}
^{p-2}\left\vert a-b\right\vert ^{2}.\label{ele_ine}%
\end{equation}
(cf. \cite{Chipot}). In particular, there is a constant $C>0$ such that
\begin{equation}
\left\vert a-b\right\vert ^{p}\leq C\left(  \left\vert a\right\vert
^{p-2}a-\left\vert b\right\vert ^{p-2}b\right)  \left(  a-b\right)
.\label{elem-ineq}%
\end{equation}
By applying (\ref{elem-ineq}) with $a=a_{j}$ and $b,$ we easily get%
\begin{equation}
\left\Vert u_{j}-u\right\Vert _{X_{0}}^{p}=B_{h_{j}}\left(  u_{j}%
-u,u_{j}-u\right)  \leq C\left(  B_{h_{j}}\left(  u_{j},u_{j}-u\right)
-B_{h_{j}}\left(  u,u_{j}-u\right)  \right)  .\label{strongX0}%
\end{equation}
By Theorem \ref{Th0} $u_{j}=\mathcal{L}_{h_{j}}^{-1}\left(  f\right)  $ and
$u=\mathcal{L}_{h}^{-1}\left(  f\right)  ,$ whence
\begin{align*}
B_{h_{j}}\left(  u_{j},u_{j}-u\right)   &  =\left\langle f,u_{j}\right\rangle
-\left\langle f,u\right\rangle \\
&  =B_{h_{j}}\left(  u_{j},u_{j}\right)  -B_{h}\left(  u,u\right)  ,
\end{align*}
which vanishes if $j\rightarrow\infty$ because of (\ref{G-con-energies2}).
Thus, the first term of the right part of (\ref{strongX0}) goes to zero. The
second term of (\ref{strongX0}) can be rewritten as%
\[
B_{h_{j}}\left(  u,u_{j}-u\right)  =%
{\displaystyle\int_{\mathbb{R}^{N}}}
{\displaystyle\int_{\mathbb{R}^{N}}}
\psi_{j}\left(  x^{\prime},x\right)  D\left(  u_{j}-u\right)  \left(
x^{\prime},x\right)  dx^{\prime}dx,
\]
where
\[
\psi_{j}\left(  x^{\prime},x\right)  =k^{1/p^{\prime}}h_{j}\frac{\left\vert
u^{\prime}-u\right\vert ^{p-2}\left(  u^{\prime}-u\right)  }{\left\vert
x^{\prime}-x\right\vert ^{p-1}},
\]
and $D\left(  u_{j}-u\right)  $ is the \textit{nonlocal gradient} of
$u_{j}-u,$
\[
D\left(  u_{j}-u\right)  \left(  x^{\prime},x\right)  =k^{1/p}\frac{\left(
u_{j}^{\prime}-u_{j}-\left(  u^{\prime}-u\right)  \right)  }{\left\vert
x^{\prime}-x\right\vert }%
\]
We also define the \textit{nonlocal divergence} of $\psi_{j}$ as
\[
d\psi_{j}\left(  x\right)  =%
{\displaystyle\int_{\mathbb{R}^{N}}}
k^{1/p}\frac{\psi_{j}\left(  x,x^{\prime}\right)  -\psi_{j}\left(  x^{\prime
},x\right)  }{\left\vert x^{\prime}-x\right\vert }dx^{\prime}.
\]
We have the following properties for these functions (see \cite{Bonder} for a
detailed account): by the anti-symmetry of $\psi_{j},$ it is clear that
\[
d\psi_{j}\left(  x\right)  =%
{\displaystyle\int_{\mathbb{R}^{N}}}
k^{1/p}\frac{\psi_{j}\left(  x,x^{\prime}\right)  -\psi_{j}\left(  x^{\prime
},x\right)  }{\left\vert x^{\prime}-x\right\vert }dx^{\prime}=-2%
{\displaystyle\int_{\mathbb{R}^{N}}}
k^{1/p}\frac{\psi_{j}\left(  x^{\prime},x\right)  }{\left\vert x^{\prime
}-x\right\vert }dx^{\prime}.
\]
Besides, the integrability property of the kernel (\ref{0}) and H\"{o}lder
inequality make automatic the proof of that the sequence $\left(  d\psi
_{j}\right)  _{j}$ is uniformly bounded in $L^{p^{\prime}}.$ At this point,
one has solely to perform a change of variables to deduce how $d\psi
_{j}\left(  x\right)  \in X_{0}^{\prime}$ acts on $u_{j}-u\in X_{0}:$ it can
be expressed through the next \textit{formula of integration by parts}:%
\begin{equation}%
{\displaystyle\int_{\mathbb{R}^{N}}}
d\psi_{j}\left(  x\right)  \left(  u_{j}-u\right)  \left(  x\right)  dx=%
{\displaystyle\int_{\mathbb{R}^{N}}}
{\displaystyle\int_{\mathbb{R}^{N}}}
\psi_{j}\left(  x^{\prime},x\right)  D\left(  u_{j}-u\right)  \left(
x^{\prime},x\right)  dx^{\prime}dx.\label{partes}%
\end{equation}

By gathering all the above comments we derive
\begin{align*}
\left\vert B_{h_{j}}\left(  u,u_{j}-u\right)  \right\vert  &  \leq\left\vert
{\displaystyle\int_{\mathbb{R}^{N}}}
d\psi_{j}\left(  x\right)  \left(  u_{j}-u\right)  \left(  x\right)
dx\right\vert \leq\left(
{\displaystyle\int_{\mathbb{R}^{N}}}
\left\vert d\psi_{j}\left(  x\right)  \right\vert ^{p^{\prime}}dx\right)
^{1/p^{\prime}}\left(
{\displaystyle\int_{\mathbb{R}^{N}}}
\left\vert u_{j}-u\right\vert ^{p}\left(  x\right)  dx\right)  ^{1/p}\\
&  \leq C\left(
{\displaystyle\int_{\mathbb{R}^{N}}}
\left\vert u_{j}-u\right\vert ^{p}\left(  x\right)  dx\right)  ^{1/p}%
\rightarrow0\text{ if }j\rightarrow\infty.
\end{align*}
Thus, we have proved that the right part of (\ref{strongX0}) goes to zero and
the desired strong convergence for a subsequence of $u_{j}$ toward $u$ in
$X_{0}$ has been checked:%
\begin{equation}
\lim_{j}%
{\displaystyle\int_{\mathbb{R}^{N}}}
{\displaystyle\int_{\mathbb{R}^{N}}}
kh_{j}\frac{\left\vert u_{j}^{\prime}-u_{j}-\left(  u^{\prime}-u\right)
\right\vert ^{p}}{\left\vert x^{\prime}-x\right\vert ^{p}}dx^{\prime}dx=0.
\label{1}%
\end{equation}
If $1<p<2$ the procedure is similar. We pay attention to this elementary
manipulation: fixed $m=\frac{\left(  2-p\right)  p}{2},$ we apply H\"{o}lder's
inequality. Then, by taking into account that $k\in L^{1}\left(
\mathbb{R}^{N}\right)  $ and using the left part of the inequality
(\ref{ele_ine}), we obtain that%
\begin{align*}
\left\Vert u_{j}-u\right\Vert _{X_{0}}^{p}  &  =%
{\displaystyle\int_{\mathbb{R}^{N}}}
{\displaystyle\int_{\mathbb{R}^{N}}}
\left\vert a_{j}-b\right\vert ^{p}\frac{\left\vert x^{\prime}-x\right\vert
^{m}}{\left\vert x^{\prime}-x\right\vert ^{m}}\frac{kh_{j}}{\left\vert
x^{\prime}-x\right\vert ^{p}}dx^{\prime}dx\\
&  \leq C\left(
{\displaystyle\int_{\mathbb{R}^{N}}}
{\displaystyle\int_{\mathbb{R}^{N}}}
\frac{\left\vert a_{j}-b\right\vert ^{2}}{\left\vert x^{\prime}-x\right\vert
^{2-p}}\frac{kH_{j}}{\left\vert x^{\prime}-x\right\vert ^{p}}dx^{\prime
}dx\right)  ^{p/2}\\
&  \leq C\left(  B_{h_{j}}\left(  u_{j},u_{j}-u\right)  -B_{h_{j}}\left(
u,u_{j}-u\right)  \right)  ^{p/2}.
\end{align*}
Thus, to finish the proof, it is enough to apply the analysis already employed
for the case $p\geq2.$
\end{proof}

\subsection{$G$-convergence}

\begin{theorem}
\label{Th2}From the given sequence of coefficients $\left(  h_{j}\right)
_{j}\subset\mathcal{H},$ we can extract a subsequence, which will not be
relabelled, such that $h_{j}\rightharpoonup h$ weak-$\ast$ in $L^{\infty
}\left(  \mathbb{R}^{N}\times\mathbb{R}^{N}\right)  $ and $\left(
\mathcal{L}_{h_{j}}\right)  _{j}$ $G$-converges to $\mathcal{L}_{h}$ if
$j\rightarrow\infty.$
\end{theorem}

The proof of $G$-convergence is unnecessary because the strong convergence in
$X_{0}$ implies weak convergence in that space. Despite this consideration and
for the sake of clarity, we will proceed with the proof. So, if we consider
any $f\in X_{0}^{\prime}$ and we fix $h\in\mathcal{H}$, then we can ensure the
existence of a unique solution $u_{f}\in X_{0}$ to the problem $\mathcal{L}%
_{h}v=f.$ Then, it is also clear that%
\[
\left\langle f,u_{j}\right\rangle _{H_{0}^{\prime}\times H_{0}}=\left\langle
\mathcal{L}_{h}u_{f},u_{j}\right\rangle _{H_{0}^{\prime}\times H_{0}}%
=B_{h}\left(  u_{f},u_{j}-u\right)  +B_{h}\left(  u_{f},u\right)  .
\]
By taking into account that%
\[
B_{h}\left(  u_{f},u_{j}-u\right)  \leq\left(  B_{h}\left(  u_{f}%
,u_{f}\right)  \right)  ^{1/p^{\prime}}\left(  B_{h}\left(  u_{j}%
-u,u_{j}-u\right)  \right)  ^{1/p}%
\]
and using Theorem \ref{Th1}, we get the desired $G$-convergence, namely
\[
\lim_{j}\left\langle f,u_{j}\right\rangle _{H_{0}^{\prime}\times H_{0}}%
=B_{h}\left(  u_{f},u\right)  =\left\langle f,u\right\rangle _{H_{0}^{\prime
}\times H_{0}}.\smallskip
\]

\quad Even though the notion of $G$-convergence is, by many reasons, not
satisfactory for the study of differential equations, it satisfies the basic
axiomatic rules (\cite{Pankov, Allaire}). For instance, the $G$-limit of a
sequence of operator is unique. To prove it, we proceed as follows: assume
that for any $f\in X_{0}^{\prime},$ $\mathcal{L}_{h}$ and $\mathcal{L}_{g}$
are two different $G$-limits of $\left(  \mathcal{L}_{h_{j}}\right)  _{j}$. If
the underlying solutions of these operators are $u_{f}=\mathcal{L}_{h}^{-1}f$
and $u_{g}=\mathcal{L}_{g}^{-1}f,$ then $u_{j}=\mathcal{L}_{h_{j}}%
^{-1}f\rightharpoonup u_{f}=\mathcal{L}_{h}^{-1}f$ and $u_{j}\rightharpoonup
u_{g}=\mathcal{L}_{g}^{-1}f$ weakly in $X_{0}.$ Then $u_{f}=u_{g},$ whereby we
deduce $B_{h}(u_{f},w)=\left\langle f,w\right\rangle $ and $B_{g}\left(
u_{f},w\right)  =\left\langle f,w\right\rangle ,$ for any $w\in X_{0}.$ Then
$B_{h-g}\left(  u_{f},w\right)  =0$ for any $w\in X_{0}$ and, in particular,
$B_{h-g}\left(  u_{f},u_{f}\right)  =0.$ Thus, since this equality holds for
any arbitrary functional $f\in X_{0}^{\prime},$ then the same equality holds
for arbitrary $u_{f}\in X_{0}.$ If we perform variations with respect to $u$
in the expression $B_{h-g}\left(  u,u\right)  =0$ we get
\begin{equation}
\int_{\Omega}\int_{\mathbb{R}^{N}}\left(  h-g\right)  k\frac{\left\vert
u^{\prime}-u\right\vert ^{p-2}\left(  u^{\prime}-u\right)  }{\left\vert
x^{\prime}-x\right\vert ^{p}}v\left(  x\right)  dx^{\prime}dx=0 \label{var}%
\end{equation}
for any $u\in X_{0}$ and any $v\in L^{\infty}\left(  \Omega\right)  $ with
$\operatorname*{supp}v\subset\Omega.$ Now, if we follow the lines marked in
the proof of Proposition 17 from \cite{Elbau}, we conclude that the function
\[
V\left(  x^{\prime},x,w,z\right)  =\left(  h\left(  x^{\prime},x\right)
-g\left(  x^{\prime},x\right)  \right)  k\left(  \left\vert x^{\prime
},x\right\vert \right)  \frac{\left\vert z-w\right\vert ^{p-2}\left(
z-w\right)  }{\left\vert x^{\prime}-x\right\vert ^{p}}%
\]
is constant in $z,$ for any $\left(  x,x^{\prime}\right)  \in\Omega
\times\mathbb{R}^{N}$ and any $w\in\mathbb{R}^{N}.$ This clearly implies
$h-g=0$ in $\left(  x^{\prime},x\right)  \in\Omega\times\mathbb{R}^{N}.$ If we
use the symmetry property of the coefficients $h$ and $g,$ then the same is
true $\mathbb{R}^{N}\times\Omega.$ Then $h=g$ in $\left(  \Omega
\times\mathbb{R}^{N}\right)  \cup\left(  \mathbb{R}^{N}\times\Omega\right)  ,$
which suffices to ensure $\mathcal{L}_{h}$ and $\mathcal{L}_{g}$ coincide.

\subsection{$H$-convergence}

\begin{theorem}
\label{Th3}From a given sequence of coefficients $\left(  h_{j}\right)
_{j}\subset\mathcal{H},$ we can extract a subsequence, which will be not
relabelled, such that $h_{j}\rightharpoonup h$ weak-$\ast$ in $L^{\infty
}\left(  \mathbb{R}^{N}\times\mathbb{R}^{N}\right)  $ and $\left(
\mathcal{L}_{h_{j}}\right)  _{j}$ $H$-converges to $\mathcal{L}_{h}$ if
$j\rightarrow\infty.$
\end{theorem}

\begin{proof}
We know $u_{j}\rightarrow u\in X_{0}$ strongly in $X_{0}$. It is immediate to
prove that the sequence of nonlocal fluxes $\left(  \Psi_{h_{j}}\right)  _{j}$
is bounded in $L^{p^{\prime}}.$ The question we raise here is the following:
since there is a subsequence of indexes j and a function $\Psi$, such that
$\Psi_{j}\rightarrow\Psi\in L^{p^{\prime}}\left(  \mathbb{R}^{N}%
\times\mathbb{R}^{N}\right)  $ weakly in $L^{p^{\prime}}\left(  \mathbb{R}%
^{N}\times\mathbb{R}^{N}\right)  $, it is then $\Psi=\Psi_{h}$? To answer to
this question we examine the limit%
\begin{equation}
I\doteq\lim_{j}\int\int_{j}\Psi_{j}\left(  x^{\prime},x\right)  G\left(
x^{\prime},x\right)  dx^{\prime}dx, \label{limit}%
\end{equation}
where $G$ is any function from $L^{p}\left(  \mathbb{R}^{N}\times
\mathbb{R}^{N}\right)  .$ If $G_{k}\rightarrow G$ strongly in $L^{p}\left(
\mathbb{R}^{N}\times\mathbb{R}^{N}\right)  $\ if $k\rightarrow\infty,$ then%
\begin{align*}
\lim_{j}%
{\displaystyle\int_{\mathbb{R}^{N}}}
{\displaystyle\int_{\mathbb{R}^{N}}}
\Psi_{h_{j}}\left(  x^{\prime},x\right)  G\left(  x^{\prime},x\right)
dx^{\prime}dx  &  =\lim_{j}%
{\displaystyle\int_{\mathbb{R}^{N}}}
{\displaystyle\int_{\mathbb{R}^{N}}}
\Psi_{h_{j}}\left(  x^{\prime},x\right)  G_{k}\left(  x^{\prime},x\right)
dx^{\prime}dx\\
&  +\lim_{j}%
{\displaystyle\int_{\mathbb{R}^{N}}}
{\displaystyle\int_{\mathbb{R}^{N}}}
\Psi_{h_{j}}\left(  x^{\prime},x\right)  \left(  G\left(  x^{\prime},x\right)
-G_{k}\left(  x^{\prime},x\right)  \right)  dx^{\prime}dx
\end{align*}
and thanks to H\"{o}lder's inequality and the uniform estimation of
$\Psi_{h_{j}}$ in $L^{p^{\prime}}$, we get%
\[%
{\displaystyle\int_{\mathbb{R}^{N}}}
{\displaystyle\int_{\mathbb{R}^{N}}}
\Psi_{h_{j}}\left(  x^{\prime},x\right)  \left(  G\left(  x^{\prime},x\right)
-G_{k}\left(  x^{\prime},x\right)  \right)  dx^{\prime}dx\leq C\left(
{\displaystyle\int_{\mathbb{R}^{N}}}
{\displaystyle\int_{\mathbb{R}^{N}}}
\left\vert G\left(  x^{\prime},x\right)  -G_{k}\left(  x^{\prime},x\right)
\right\vert ^{p}dx^{\prime}dx\right)  ^{1/p}\rightarrow0
\]
if $k\rightarrow\infty$ uniformly in $j.$ So, in the analysis of
(\ref{limit}), we only shall refer to the case of function $G\in C_{c}\left(
\mathbb{R}^{N}\times\mathbb{R}^{N}\right)  .$ We firstly consider the case
$p\geq2$. We perform the decomposition%
\[
I=I_{1}+I_{2}%
\]
where%
\begin{align*}
I_{1}  &  =%
{\displaystyle\int_{\mathbb{R}^{N}}}
{\displaystyle\int_{\mathbb{R}^{N}}}
k^{1/p^{\prime}}h_{j}\left(  \frac{\left\vert u_{j}^{\prime}-u_{j}\right\vert
^{p-2}\left(  u_{j}^{\prime}-u_{j}\right)  }{\left\vert x^{\prime
}-x\right\vert ^{p-1}}-\frac{\left\vert u^{\prime}-u\right\vert ^{p-2}\left(
u^{\prime}-u\right)  }{\left\vert x^{\prime}-x\right\vert ^{p-1}}\right)
G\left(  x^{\prime},x\right)  dx^{\prime}dx,\\
I_{2}  &  =%
{\displaystyle\int_{\mathbb{R}^{N}}}
{\displaystyle\int_{\mathbb{R}^{N}}}
h_{j}k^{1/p^{\prime}}\frac{\left\vert u^{\prime}-u\right\vert ^{p-2}\left(
u^{\prime}-u\right)  }{\left\vert x^{\prime}-x\right\vert ^{p-1}}G\left(
x^{\prime},x\right)  dx^{\prime}dx.
\end{align*}
The inequality (\ref{ele_ine}) yields ($\frac{1}{p^{\prime}}-\frac{1}{p}%
=\frac{p-2}{p})$
\begin{equation}
I_{1}\leq C%
{\displaystyle\int_{\mathbb{R}^{N}}}
{\displaystyle\int_{\mathbb{R}^{N}}}
h_{j}k^{\frac{1}{p}}\frac{\left\vert u_{j}^{\prime}-u_{j}-\left(  u^{\prime
}-u\right)  \right\vert }{\left\vert x^{\prime}-x\right\vert }\frac{\left(
k^{\frac{1}{p}}\left\vert u_{j}^{\prime}-u_{j}\right\vert +k^{\frac{1}{p}%
}\left\vert u^{\prime}-u\right\vert \right)  ^{p-2}}{\left\vert x^{\prime
}-x\right\vert ^{p-2}}\left\vert G\right\vert dx^{\prime}dx
\label{theinequality}%
\end{equation}
and by invoking H\"{o}lder inequality we get%
\begin{equation}%
\begin{array}
[c]{c}%
\displaystyle I_{1}\leq C\left(
{\displaystyle\iint_{\mathbb{R}^{N}\times\mathbb{R}^{N}}}
kh_{j}\frac{\left\vert u_{j}^{\prime}-u_{j}-\left(  u^{\prime}-u\right)
\right\vert ^{p}}{\left\vert x^{\prime}-x\right\vert ^{p}}dx^{\prime
}dx\right)  ^{1/p}\times\smallskip\\
\displaystyle\times\left(
{\displaystyle\iint_{\mathbb{R}^{N}\times\mathbb{R}^{N}}}
h_{j}\frac{\left\vert \left\vert u_{j}^{\prime}-u_{j}\right\vert k^{\frac
{1}{p}}+k^{\frac{1}{p}}\left\vert u^{\prime}-u\right\vert \right\vert
^{\left(  p-2\right)  p^{\prime}}\left\vert G\right\vert ^{p^{\prime}}%
}{\left\vert x^{\prime}-x\right\vert ^{\left(  p-2\right)  p^{\prime}}%
}dx^{\prime}dx\right)  ^{1/p^{\prime}}%
\end{array}
\label{theinequality2}%
\end{equation}
Since $0<\frac{\left(  p-2\right)  p}{p-1}\leq p$ and $\operatorname*{supp}%
G\subset K$, a compact set in $\mathbb{R}^{N}\times\mathbb{R}^{N}$, we can
apply Jensen's inequality to have the following uniform estimation:%
\begin{align*}
&
{\displaystyle\int_{\mathbb{R}^{N}}}
{\displaystyle\int_{\mathbb{R}^{N}}}
h_{j}\frac{\left\vert \left\vert u_{j}^{\prime}-u_{j}\right\vert k^{\frac
{1}{p}}+k^{\frac{1}{p}}\left\vert u^{\prime}-u\right\vert \right\vert
^{\frac{\left(  p-2\right)  p}{p-1}}}{\left\vert x^{\prime}-x\right\vert
^{\frac{\left(  p-2\right)  p}{p-1}}}\left\vert G\right\vert ^{p^{\prime}%
}dx^{\prime}dx\\
&  \leq C%
{\displaystyle\iint_{K}}
h_{j}k\frac{\left\vert \left\vert u_{j}^{\prime}-u_{j}\right\vert +\left\vert
u^{\prime}-u\right\vert \right\vert ^{p}}{\left\vert x^{\prime}-x\right\vert
^{p}}dx^{\prime}dx\\
&  \leq C%
{\displaystyle\int_{\mathbb{R}^{N}}}
{\displaystyle\int_{\mathbb{R}^{N}}}
h_{j}k\frac{\left(  \left\vert u_{j}^{\prime}-u_{j}\right\vert ^{p}+\left\vert
u^{\prime}-u\right\vert ^{p}\right)  }{\left\vert x^{\prime}-x\right\vert
^{p}}dx^{\prime}dx\leq C
\end{align*}
After (\ref{theinequality}) and (\ref{theinequality2}), is suffices to use
(\ref{1}) to deduce $\lim I_{1}=0.$\newline When $1<p<2$ the situation does
not essentially change because after using (\ref{ele_ine}), the estimation%
\[
\frac{\left\vert u_{j}^{\prime}-u_{j}-\left(  u^{\prime}-u\right)  \right\vert
}{\left\vert x^{\prime}-x\right\vert }\frac{\left(  \left\vert u_{j}^{\prime
}-u_{j}\right\vert +\left\vert u^{\prime}-u\right\vert \right)  ^{p-2}%
}{\left\vert x^{\prime}-x\right\vert ^{p-2}}\leq\frac{\left\vert u_{j}%
^{\prime}-u_{j}-\left(  u^{\prime}-u\right)  \right\vert ^{p-1}}{\left\vert
x^{\prime}-x\right\vert ^{p-1}}%
\]
and H\"{o}lder's inequality guarantee%
\[
I_{1}\leq C%
{\displaystyle\int_{\mathbb{R}^{N}}}
{\displaystyle\int_{\mathbb{R}^{N}}}
h_{j}k^{\frac{1}{p^{\prime}}}\frac{\left\vert u_{j}^{\prime}-u_{j}-\left(
u^{\prime}-u\right)  \right\vert ^{p-1}}{\left\vert x^{\prime}-x\right\vert
^{p-1}}dx^{\prime}dx\leq C%
{\displaystyle\int_{\mathbb{R}^{N}}}
{\displaystyle\int_{\mathbb{R}^{N}}}
h_{j}k\frac{\left\vert u_{j}^{\prime}-u_{j}-\left(  u^{\prime}-u\right)
\right\vert ^{p}}{\left\vert x^{\prime}-x\right\vert ^{p}}dx^{\prime}dx,
\]
and again, thanks to (\ref{1}), we deduce $\lim I_{1}=0.$\newline Now we look
at the limit of $I_{2}.$ We know $h_{j}\rightharpoonup h$ weak-$\ast$ in
$L^{\infty}\left(  \mathbb{R}^{N}\times\mathbb{R}^{N}\right)  $, $k_{\delta
}^{1/p^{\prime}}\frac{\left\vert u^{\prime}-u\right\vert ^{p-2}\left(
u^{\prime}-u\right)  }{\left\vert x^{\prime}-x\right\vert ^{p-1}}\in
L^{p^{\prime}}$ and $G\in L^{\infty},$ then it is evident that $k^{1/p^{\prime
}}\frac{\left\vert u^{\prime}-u\right\vert ^{p-2}\left(  u^{\prime}-u\right)
}{\left\vert x^{\prime}-x\right\vert ^{p-1}}G\left(  x^{\prime},x\right)  \in
L^{1}\ $and thereby,%
\[
\lim I_{2}=%
{\displaystyle\int_{\mathbb{R}^{N}}}
{\displaystyle\int_{\mathbb{R}^{N}}}
hk^{1/p^{\prime}}\frac{\left\vert u^{\prime}-u\right\vert ^{p-2}\left(
u^{\prime}-u\right)  }{\left\vert x^{\prime}-x\right\vert ^{p-1}}G\left(
x^{\prime},x\right)  dx^{\prime}dx
\]
We have proved $\Psi_{h_{j}}\rightharpoonup\Psi_{h}$ weakly in $L^{p^{\prime}%
}\left(  \mathbb{R}^{N}\times\mathbb{R}^{N}\right)  $ if $j\rightarrow\infty.$
\end{proof}

We have proved strong convergence in $X_{0}$ implies $H$ convergence. The
reverse implication is also true. Indeed, if we assume $u_{j}=\mathcal{L}%
_{h_{j}}f\rightharpoonup u=\mathcal{L}_{h}f\ $weakly in $X_{0}$ and
$\Psi_{h_{j}}\rightharpoonup\Psi_{h}$ weakly in $L^{p^{\prime}}$, then Theorem
\ref{Th0} factually establishes that $u_{j}\rightarrow u$ strongly in $L^{p}$,
at least for a subsequence. In addition, the weak convergence of $\left(
\Psi_{h_{j}}\right)  _{j}$ in $L^{p^{\prime}}$ ensures this sequence is
uniformly bounded $L^{p^{\prime}}$. Then, by using (\ref{partes}) we realize
that the norm of $\left\Vert u_{j}-u\right\Vert _{X_{0}}\ $ tends to zero:%
\begin{align*}
\lim_{j}B_{h}\left(  u_{j}-u,u_{j}-u\right)   &  =\lim_{j}\int d\Psi
_{j}\left(  x\right)  \left(  u_{j}-u\right)  \left(  x\right)  dx\\
&  \leq\lim_{j}\left\Vert d\Psi_{j}\left(  x\right)  \right\Vert
_{L^{p^{\prime}}}\left\Vert u_{j}-u\right\Vert _{L^{p}}=0
\end{align*}
This amounts to state the strong convergence for the whole sequence $\left(
u_{j}\right)  _{j},$ towards $u$ strongly in $X_{0}.$

\section{Conclusions\label{S5}}

Our starting point is the nonlocal energy criterion, Theorem \ref{Th0}. This
result establishes $u_{j}=\mathcal{L}_{h_{j}}^{-1}f\rightarrow u=\mathcal{L}%
_{h}^{-1}f$ strongly in $L^{p}$ (at least for a subsequence). After that, we
have improved this convergence by means of Theorem \ref{Th1}. We have proved
$u_{j}=\mathcal{L}_{h_{j}}^{-1}f\rightarrow u=\mathcal{L}_{h}^{-1}f$ strongly
in $X_{0},\ $at least for a subsequence. From this strong convergence in
$X_{0}$ we have inferred, as a particular case, the $G$-convergence of
$\left(  \mathcal{L}_{h_{j}}\right)  _{j}$ towards $\mathcal{L}_{h}$ (Theorem
\ref{Th2}). And finally, we have proved that the $H$-convergence is equivalent
to the strong convergence in $X_{0}$ (see Theorem \ref{Th3} and the final comments).

We conclude that, as it occurs in the local setting, the appropriate
conditions of monotonicity for the nonlocal operator and the use of the
corresponding nonlocal energy criterion seem to be the key points to establish
the type of convergences we have studied here.

\section*{Acknowledgement}

This work was supported by the Spanish Project MTM2017-87912-P, Ministerio de
Econom\'{\i}a, Industria y Competitividad (Spain), and by the Regional Project
SBPLY/17/180501/000452, JJ. CC. de Castilla-La Mancha. There are no conflicts
of interest to this work.

\end{document}